\documentclass[10pt]{article}

\usepackage[utf8]{inputenc}
\usepackage{amsthm,amsmath,amssymb}
\usepackage{tikz}
\usepackage{graphicx,subcaption,caption}
\usepackage{algorithm,algorithmic}
\usepackage{pgfplots}
\pgfplotsset{width=7cm,compat=1.3}
\usepgfplotslibrary{external}
\newtheorem{theorem}{Theorem}
\newtheorem{proposition}{Proposition} 
\newtheorem{lemma}{Lemma}
\newtheorem{assumption}{Assumption}
\newtheorem{remark}{Remark}
\newtheorem{corollary}{Corollay}
\newtheorem{definition}{Definition}
\usepackage{hyperref}

\begin{document}
\title{Global Optimum is not Limit Computable}
\author{K. Lakshmanan \thanks{Department of Computer Science and Engineering, Indian Institute of Technology (BHU), Varanasi 221005, Email: lakshmanank.cse@iitbhu.ac.in }}
\date{}

\maketitle


\begin{abstract}
We study the limit computability of finding a global optimum of a non-convex continuous function. We give a short proof to show that the problem of checking whether a point is a global minimum is not limit computable. Thereby showing the same for the problem of finding a global minimum. In the next part, we give an algorithm that converges to the global minima when a lower bound on the size of the basin of attraction of the global minima is known. We prove the convergence of this algorithm and provide some numerical experiments.
\end{abstract}

\section{Introduction and Preliminaries}
We consider the problem of finding the global minima of a non-convex continuous function $f:  C \rightarrow \mathbb{R}$, where $ C \subset \mathbb{R}^n$ is a closed, compact subset. Global minima is the point $x^* \in C$ which satisfies the following property: $f(x^*) \leq f(x)$ for all $x \in C$. The function $f$ attains this minimum at least once by extreme value theorem. Our goal is to find one such point. This problem is well-studied with many books written on the subject, see for example \cite{globook}.

In this paper, we show that this problem is not limit computable. That is there is no algorithm that can convergence to the global minima of any continuous function without any knowing any other information about the function. In fact, we show a simpler problem of checking whether a local minimum is global is itself not limit computable. Next with the knowledge about the basin of attraction of the global minima we give a fast algorithm that converges to the global minima. Before proceeding further we define some preliminary concepts about reducibility and limit computability. These definitions are as in chapters 1 and 3 of \cite{sorbook}.


\begin{definition}
A Turing machine has a two-way infinite tape divided into cells, a reading head which scans one cell of the tape at a time, and a finite set of internal states $Q=\{q_0,q_1,\ldots,q_n\}, \, n \geq 1$. Each cell is either blank or has symbol 1 written on it. In a single step the machine may simultaneously (1) change the from one state to another; (2) change the scanned symbol $s$ to another symbol $s'\in S = \{1,B\}$; (3) move the reading head one cell to the right (R) or left (L). This operation of machine is controlled by a partial map $\delta : Q \times S \rightarrow Q \times S \times \{R,L\}$.
\end{definition}

The map $\delta$ viewed as a finite set of quintuples is called a Turing program.

\begin{definition}
 An oracle Turing machine (o-machine) is a Turing machine with an extra `` read only " tape called the oracle tape, upon which the characteristic function of some set $A$ is written. Two reading heads move along these two tapes simultaneously.
\end{definition}

In a given state $q_1$ if the tapes contain symbols $s_1$ and $t_1$ the machine changes the work tape symbol to $t_2$ change the state to $q_2$ and move the head either right or left independently. An oracle Turing machine is a finite sequence of program lines. Fix an effective coding (G\"{o}del numbering) of all oracle Turing programs for o-machines. Let $\tilde{P}_e$ denote the $e$th such oracle program under this effective coding. If the oracle machine halts, let u be the maximum cell on the oracle tape scanned during computation, i.e., maximum integer whose membership in $A$ has been tested. We say that the elements $z \leq u$ are used in computation.  If no element is scanned we let $u = 0$.

\begin{definition}
 If the oracle program $\tilde{P}_e$ with $A$ on the oracle tape and input $x$ halts with output $y$ and if $u$ is the maximum element used on the oracle tape during computation, then we write
 \[\Phi^A_e(x) = y \mbox{ and } \varphi^A_e(x)=u. \]
 We refer $\Phi^A_e(x)$ as a Turing functional and we call corresponding $\varphi^A_e(x)$ the corresponding use function. The functional is determined by the program $\tilde{P}_e$ and may be partial or total.
\end{definition}

\begin{definition}
 A partial function $\theta$ is Turing computable in A (A-Turing computable), written $\theta \leq_T A$, if there is an $e$ such that $\Phi^A_e(x) \downarrow = y$ if and only if $\theta(x) = y$. A set $B$ is Turing reducible to $A$ ($B \leq_T A$) if the characteristic function $\chi_B \leq_T A$.
\end{definition}

We denote the set $\{0,1,2,3,\ldots \}$ by $\omega$.

\begin{definition}
 A set $A$ is limit computable if there is a computable sequence $\{A_s\}_{s \in \omega}$ such that for all $x$,
 \[ A(x) = \lim_s A_s(x). \]
 Here $A(x)$ is the characteristic function of $A$.
\end{definition}

\begin{definition}
A set $A$ is $\Sigma_2$ if there is a computable relation $R$ with 
\[ x \in A \Leftrightarrow (\exists y) (\forall z) R(x,y,z) .\]
A set is $\Pi_2$ if $\bar{A}$ is in $\Sigma_2$. And a set A is $\triangle_2$ if $A \in \Sigma_2$ and $A \in \Pi_2$. 
\end{definition}

\begin{lemma}[Limit Lemma (Shoenfield 1959)]
 A set $A$ is limit computable if and only if $A \in \triangle_2$.
\end{lemma}

\begin{proof}
 We refer to proof of lemma 3.6.2 of \cite{sorbook}.
\end{proof}
By Limit lemma we also call $\{A_s\}_{s \in \omega}$ the $\triangle_2$-approximation of $A$.

\begin{definition}
 Let $A \upharpoonleft \upharpoonleft x$ denote the set $\{A(y): y \leq x \}$. Given $\{A_s\}_{s \in \omega}$, any function $m(x)$ is a modulus (of convergence) if 
 \[ \forall x (\forall s \geq m(x)) \, [A \upharpoonleft \upharpoonleft x = A_s \upharpoonleft \upharpoonleft x ]. \]
\end{definition}

\begin{proposition}\label{keyprop}
 If $A$ is limit computable via $\{A_s\}_{s \in \omega}$ with any modulus $m(x)$, then $A \leq_T m$.
\end{proposition}

\begin{proof}
 Take $A(x) = A_{m(x)}(x)$.
\end{proof}

\section{Main Theorem}
Let the set of global minima of the function $f$ be denoted by $G$. We have the following lemma for the set $G$.

\begin{lemma}\label{mainlem}
 The set $G$ is not limit computable.
\end{lemma}

\begin{proof}
 Now consider the modified problem where we define $h_z(x) := \min \{0,f(z)-f(x)\}$. This function is identically zero if and only if $z = x^*$. Hence our problem of finding the global minimum is same as checking whether $h_z(\cdot)$ is identically zero. Since our objective function $f$ is continuous, $h_z(\cdot)$ is also continuous. An example function is shown in figure \ref{illusfig}. The plot on the left shows the original objective function and on right shows the modified function which has to be checked if identically zero. Since $G$ is the set of all global minima it is also the set of all points $z$ where the function $h_z(\cdot)$ is identically zero.

\begin{figure}[!ht]
\centering
\includegraphics{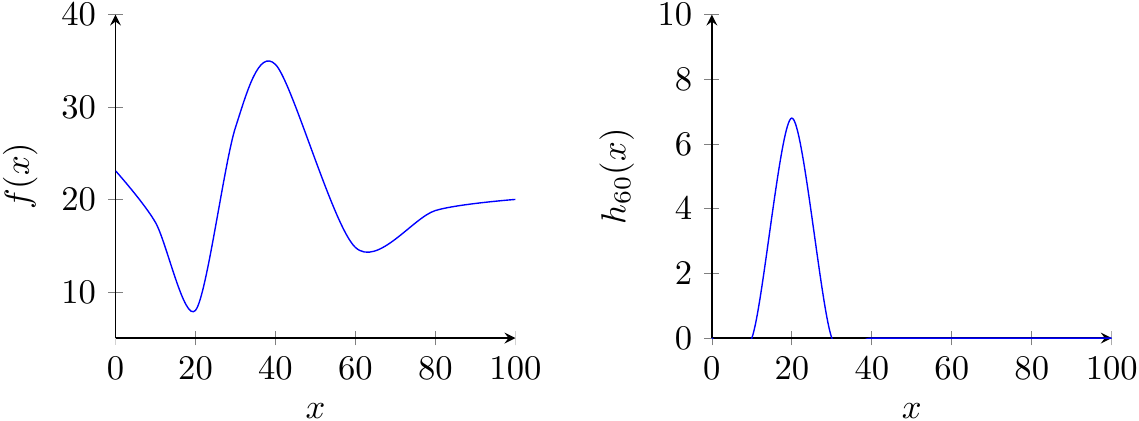}
\caption{The figure on the right shows a sample objective function. The figure of the left is the corresponding function $h_{60}(x) := \min \{0,f(60)-f(x) \}$  as in the proof of Lemma ~\ref{mainlem}. The function $h_{60}(x)$ is not identically zero as $x=60$ is not the global minima of $f(x)$.}
\label{illusfig}
\end{figure}
 
  For $G$ to be limit computable via $\{G_s\}_{s \in \omega}$ with some modulus function $m$ we need,
 \begin{align*}
 \forall x' (\forall s \geq m(x')) & \, [G \upharpoonleft \upharpoonleft x' = G_s \upharpoonleft \upharpoonleft x' ] \\
 \forall x' (\forall s \geq m(x')) & \, \{ G(z): z \leq x' \} = \{ G_s(z): z \leq x' \}.
 \end{align*}
  Note that we need $\{G_s\}_{s \in \omega}$ to be computable. But to compute whether $G_s(z) = 0$ or not for all $z \leq x'$ involves checking whether $h_z(x) \equiv 0 $. That is for each point $z$ we have to check a function is identically zero. But this cannot be checked for any particular $z$ unless it is checked for all $x$. As the function $h_z(x)$ is real valued and continuous if it is non-zero, there exists an interval $I$ whose length can be arbitrarily small such that the function is non-zero in this interval. Since we do not know the length of this interval we have to check for all points which is dense in $x$. But this set of points $x$ is not finite.
  
  As no Turing machine can check (halt) if a function is zero at infinitely many points. We see that 
  there is no computable $\{G_s\}_{s \in \omega}$ with some modulus function $m$ such that
  \[ G(x) = \lim_s G_s(x). \]
  That is we have shown that the set $G$ is not limit computable.
 \end{proof}

 \begin{corollary}
  The problem of checking whether local minima $z$ is global is also not limit computable as this involves checking whether $h_z(\cdot)$ is identically zero.
 \end{corollary}

 \begin{corollary}
  By Limit lemma we can in fact show that the set $G$ of global minima is in $\Pi_2$ but not in $\Sigma_2$ as there is an oracle-Turing machine which will halt and produce the right output if the function $h_z(\cdot)$ is not identically zero. But not when the function is identically zero.
 \end{corollary}

 Now we can state the main theorem.
 
 \begin{theorem}\label{mainthm}
  Finding global minima of a continuous function is not limit computable.
 \end{theorem}

 \begin{proof}
  Suppose finding the global minima is limit computable then have an oracle machine for computing the set of global minima. But this contradicts the Lemma \ref{mainlem} that set of global minima is not limit computable. 
 \end{proof}
 
 \begin{remark}
  By definition a point $x \in A$ is a local minima if $(\exists n \in \mathbb{N}^+)(\forall y \in B(y,1/n))$  $f(x) < f(y)$. Here $B(y,1/n)$ is the neighbourhood of $y$ with radius $1/n$. Take this to be the computable relation $R$ in the Limit lemma i.e., we have $x \in A \Leftrightarrow (\exists n)( \forall y) \, R(x,n,y)$. Thereby we get that the computing local minima is limit computable.
 \end{remark}

 When additional information is known about the global minima, like it's basin of attraction then the global minima may be limit computable. In fact we give an algorithm converging to the global minima when the basin of attraction is known in the next section.
 
\section{An Algorithm when Basin of Attraction is Known}
In this section, we also assume the function $f$ to be differentiable. Let us denote the gradient by $\triangledown f(x)$. The algorithm takes as input the lower bound $m$ on the basin of attraction of the global minimizer. By basin of attraction we mean the following: if we let the initial point to be in the hypercube of length $m$ in all co-ordinates, i.e., in the basin of attraction around the global minima then the gradient descent algorithm will converge to the minima. The algorithm finds the point $z_k$ where the function takes a minimum amongst all points at a distance of $m$ from each other and does a gradient descent step from the point $z_k$. In this algorithm for simplicity, we do not consider line searches and use constant step-size $t > 0$. The figure \ref{figalg} shows the gradient descent step taken at the point which has the minimum function value amongst all the points in the grid.

We note the similarity of our algorithm with the one considered in paper \cite{helon}, where the basin of attraction of global minimizer is first found by searching then a gradient descent is performed. In our algorithm, these two steps are interleaved. The major issue with their algorithm is that they assume the value of the global minima is known which they assume to be zero. But this need not be known in real-world problems. This assumption is not needed with our approach. Moreover, we have formally shown the convergence of our algorithm.

\begin{figure}[!ht]
\centering
\includegraphics{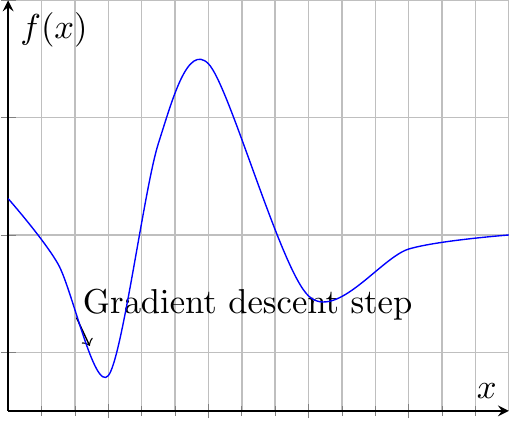}
\caption{The function $f$ to be minimized. Gradient descent step is shown for the interval where the function value is minimum. This interval is a subset of the basin of attraction of global minima.}
\label{figalg}
\end{figure}

\begin{algorithm}[!ht]
\caption{Global Optimization Algorithm \\ Input: Function $f$ and a lower bound $m$ on length of a hypercube contained in basin of attraction of global minimizer of $f$}
\begin{algorithmic}[1]
\STATE Let $C=[a,b]^d$. For simplicity we let the interval $[a,b]$ to be the same in all dimensions.
\STATE Set $y_0 = [a_1^0,\ldots,a_d^0]$ where $a_i^0 = a$ for all $i=1,\ldots,d$. And set $y_j = y_{j-1} + m$ for $j=1,\ldots,(b-a)/m$. Let $x_0 = z_0 = \arg\min_{j=0,\ldots,(b-a)/m} \{f(y_j)\}$.
\WHILE{$k=1,\ldots, \mathcal{L}$}
\STATE Set $y_0 = [a_1^k,\ldots,a_d^k]$ where $a_i^k = a_i^{k-1} - t \triangledown f(z_k)$ for all $i=1,\ldots,d$. And set $y_j = y_{j-1} + m$ for $j=1,\ldots,(b-a)/m$.
\STATE As before let $z_k = \arg\min_{j=0,\ldots,(b-a)/m} \{f(y_j)\}$.
\STATE Update $x_{k} = z_k - t \triangledown f(z_k)$.
\STATE $k = k+1$
\ENDWHILE
\end{algorithmic}
\end{algorithm}

\section{Convergence Analysis}
We show the convergence of the algorithm given in the preceding section. We make the following assumption.

\begin{assumption}\label{ass1}
 The function $f$ is twice differentiable. The gradient of $f$ is Lipschitz continuous with constant $0 < L < 1$, i.e., \[ \parallel \triangledown f(x) -\triangledown f(z) \parallel_2 \leq L \parallel x -z \parallel_2. \]
 That is we have $\triangledown^2 f(x) \preceq LI.$
\end{assumption}

We first state the following lemma used in the proof of the convergence theorem.

\begin{lemma} \label{boundlem}
 Assume that the function $f$ satisfies Assumption \ref{ass1} and the step-size $t \leq 1/L$. We also assume that the global minima $x^*$ is unique. Then there exists a constant $R > 0$ such that for all balls $B(x^*,r)$ with radius $r < R$, there is a $M_{r} > 0$ and that the iterates of the algorithm $\{x_k\}$ remains in this ball $B(x^*,r)$ asymptotically, i.e., $x_k \in B(x^*,r)$ for $k \geq M_{r}$.
\end{lemma}

\begin{proof}
From assumption \ref{ass1} we have that $\triangledown^2 f(x) - LI$ is negative semi-definite matrix. Using a quadratic expansion of $f$ around $f(x^*)$, we obtain the following inequality for $x \in B(x^*,r)$
\begin{align}
 \nonumber f(x) &\leq f(x^*) + \triangledown f(x^*)^T (x-x^*) + \frac{1}{2} \triangledown^2 f(x^*) \parallel x-x^* \parallel_2^2 \\
 f(x) &\leq f(x^*) + \frac{1}{2} L \parallel x-x^* \parallel_2^2 \label{basbound}
\end{align}
Since $x^*$ is a global minima we have $f(x^*) \leq f(x)$ for all $x \in C$. Let $\tilde x$ be any local minima which is not global minima. Hence $f(\tilde x) =  f(x^*) + \delta_{\tilde x} $. Now let $\delta = \min_{\tilde x} \delta_{\tilde x}$. Since $\tilde x$ is local minima but not global minima we have $\delta > 0$. Take $R > 0$ such that for any $x \in B(x^*,R)$, 
\[ \frac{L}{2}\parallel x - x^* \parallel_2^2 \leq \frac{\delta}{2} \] or that
$R \leq \frac{\delta}{L}.$
Now we have from equation \eqref{basbound}
\[ f(x) \leq f(x^*) + \frac{\delta}{2}, \]
for $x \in B(x^*,R)$. That is we have shown there exists a $R > 0$ such that for any $x \in B(x^*,R)$, 
\begin{equation}
 f(x) \leq f(\tilde x). \label{ballremain}
\end{equation}
Now we observe the following:
\begin{enumerate}
 \item from equation \eqref{ballremain} we can see that no other local minima can have a value $f(\tilde x)$, lower than the function value in this ball $B(x^*,R)$
 \item for sufficiently small step-size $t \leq 1/L$, the function value decreases with each gradient step (see equation \eqref{fybound} in proof of Theorem \ref{linconvthm})
\end{enumerate}
That is if $x \in B(x^*,R)$, the iterates in the algorithm can not move to another hypercube around some local minima $\tilde x$. Or that for all $r < R$ there exists $M_{r} > 0$ such that for $k \geq M_{r}$ the iterates remain in the ball $B(x^*,r)$ around $x^*$. 
\end{proof}

\begin{theorem}
Let $x^*$ be the unique global minimizer of the function $f$. We have for the iterates $\{x_k \}$ generated by the algorithm \[ \lim_{k \rightarrow \infty} f(x_k) = f(x^*). \]
\end{theorem}

\begin{proof}
Now from Lemma \ref{boundlem} we have $R > 0$ such that for all $r < R$ there exists $M_{r} > 0$ with $x_k \in B(x^*,r)$ for $k \geq M_{r}$. From the algorithm we also know that the function value decreases with each iteration. Thus we see that the sequence $\{f(x_k)\}$ converges as it is monotonic and bounded. Take a sufficiently small $r < R$, such that $B(x^*,r)$ lies in the basin of attraction. Hence we also have that $\lim_{k \rightarrow \infty} f(x_k) = f(x^*)$ as in the basin of attraction around the global minima the gradient descent converges to the minima.
\end{proof}

\begin{theorem}\label{linconvthm}
 Let $x^*$ be the unique global minimizer of the function $f$. For simplicity denote $M = M_{r}$. Let step-size $t \leq 1/L$ where  $L$ is Lipschitz constant of the gradient function in Assumption \ref{ass1}. If we also assume that the function is convex in the ball $B(x^*,r)$ we can show that at iteration $k > M$, $f(x_k)$ satisfies
 \[ f(x_k) - f(x^*) \leq \frac{\parallel x_{M} - x^* \parallel_2^2}{2 t (M-k)} .\]
 That is the gradient descent algorithm converges with rate $O(1/k)$.
\end{theorem}

\begin{proof}
Consider the gradient descent step $x_{k+1} = z_k - t \triangledown f(z_k)$ in the algorithm. Since the iterates remain in a ball around a global minima asymptotically, we have from Lemma \ref{boundlem} for $k \geq M_{r}, \,$ $z_k = x_k$. Now let $y = x - t \triangledown f(x)$, we then get:
\begin{align*}
f(y) &\leq f(x) + \triangledown f(x)^T (y-x) + \frac{1}{2} \triangledown^2 f(x) \parallel y-x \parallel_2^2 \\
&\leq f(x) + \triangledown f(x)^T (y-x) + \frac{1}{2} L \parallel y-x \parallel_2^2 \\
&= f(x) + \triangledown f(x)^T (x - t \triangledown f(x) - x) + \frac{1}{2} L \parallel y -x \parallel_2^2 \\
&= f(x) -  t \parallel \triangledown f(x) \parallel_2^2 + \frac{1}{2} L \parallel  y - x\parallel_2^2 \\
&= f(x) - \big(1-\frac{1}{2} Lt\big)t \parallel \triangledown f(x) \parallel_2^2.
\end{align*}
Using the fact that $t \leq 1/L$, $-\big(1-\frac{1}{2} Lt\big) \leq -\frac{1}{2}$, hence we have
\begin{equation} \label{fybound}
 f(y) \leq f(x) - \frac{1}{2}t\parallel \triangledown f(x) \parallel_2^2.
\end{equation}
Next we bound $f(y)$ the objective function value at the next iteration in terms of $f(x^*)$. 
Note that by assumption $f$ is convex in the ball $B(x^*,r)$. 
Thus we have for $x\in B(x^*,r)$,
\begin{align*}
 f(x) \leq f(x^*) + \triangledown f(x)^T (x-x^*)
\end{align*}
Plugging this into equation \eqref{fybound} we get,
\begin{align*}
f(y) &\leq f(x^*) + \triangledown f(x)^T (x - x^*) - \frac{t}{2} \parallel \triangledown f(x) \parallel_2^2 \\
f(y) - f(x^*) &\leq \frac{1}{2t}\bigg( 2t\triangledown f(x)^T (x - x^*) - t^2 \parallel \triangledown f(x) \parallel_2^2 \bigg) \\
f(y) - f(x^*) &\leq \frac{1}{2t}\bigg( 2t\triangledown f(x)^T (x - x^*) - t^2 \parallel \triangledown f(x) \parallel_2^2 \\
&\quad \quad \quad \quad \quad \quad - \parallel x - x^* \parallel_2^2 + \parallel x - x^* \parallel_2^2 \bigg) \\
f(y) - f(x^*) &\leq \frac{1}{2t}\bigg( \parallel x - x^* \parallel_2^2 - \parallel x - t\triangledown f(x) -  x^* \parallel_2^2 \bigg)
\end{align*}
By definition we have $y = x - t \triangledown f(x)$, plugging this into the previous equation we have
\begin{equation}
f(y) - f(x^*) \leq \frac{1}{2t}\bigg( \parallel x - x^* \parallel_2^2 - \parallel y -  x^* \parallel_2^2 \bigg) 
\end{equation}
This holds for all gradient descent iterations $i \geq M$. Summing over all such iterations we get:
\begin{align*}
 \sum_{i=M}^k \big(f(x_i) - f(x^*)\big) &\leq \sum_{i=M}^k \frac{1}{2t} \bigg(\parallel x_{i-1} - x^* \parallel_2^2 - \parallel x_i - x^* \parallel_2^2 \bigg) \\
 &= \frac{1}{2t} \bigg(\parallel x_{M} - x^* \parallel_2^2 - \parallel x_k - x^* \parallel_2^2 \bigg) \\
 &\leq \frac{1}{2t} \bigg(\parallel x_{M} - x^* \parallel_2^2 \bigg).
\end{align*}
Finally using the fact that $f$ is decreasing in every iteration, we conclude that
\[ f(x_k) - f(x^*) \leq \frac{1}{k} \sum_{i=M}^k \big(f(x_i) - f(x^*)\big) \leq \frac{1}{2t(M-k)} \parallel x_{M} - x^* \parallel_2^2 . \]
\end{proof}

\begin{remark}
 If the global minima $x^*$ is not unique, then the algorithm can oscillate around different minima. If we assume that the function is convex in a small interval around all these global minima, then we can show that the algorithm converges to one of the minimum points $x^*$. In addition like in the previous theorem we can also show that the convergence is linear.
\end{remark}

\begin{remark}
We have not considered momentum based acceleration methods which fasten the rate of convergence in this paper.
\end{remark}

\section{Experimental Results}
We present some numerical results. We tested the algorithm on standard benchmark functions shown in Tables \ref{benchf} and \ref{benchfs}. We show the plots of the function value as iteration proceeds for each of these functions. For Rastrigin, sphere and Rosenbrock functions the dimension was set to 20. We see from these plots that the algorithm converges to the optimum for each of these functions as expected. Table \ref{values} gives the step-sizes and lower bound on the basin of attraction set used for each of these functions.

\begin{table}[ht]
\centering
\caption{Various Benchmark Functions for Global Optimization}
\label{benchf}
  \begin{tabular}{|c|c|}
    \hline
    Name & Formula \\ 
    \hline
    Rastrigin Function & $f(x) = An + \sum_{i=1}^n \big(x_i^2 - A \cos(2 \pi x_i) \big)$ where $A = 10$ \\
    \hline
    Ackley Function & $\begin{aligned} &f(x) = -20 \exp \big( -0.2\sqrt{0.5(x^2 + y^2)} \big) \\ &- \exp(0.5(\cos(2\pi x) + \cos(2\pi y))) + e + 20\end{aligned}$ \\
    \hline
    Sphere Function & $f(x) = \sum_{i=1}^n x_i^2$ \\
    \hline
    Rosenbrock Function & $f(x) = \sum_{i=1}^{n-1} \big(100(x_{i+1} - x_i^2)^2 + (1-x_i)^2 \big)$ \\
    \hline
    Beale Function & $\begin{aligned} f(x) &= (1.5 - x + xy)^2 + (2.25 -x + xy^2)^2 \\ &+ (2.625 - x + xy^3 )^2 \end{aligned} $ \\
    \hline
    Booth Function & $f(x,y) = (x+2y-7)^2 + (2x+y-5)^2$ \\
    \hline
    \end{tabular}
\end{table}

\begin{table}[ht]
\centering
\caption{Global Minimum and Search Domain for these Benchmark Functions}
\label{benchfs}
  \begin{tabular}{|c|c|c|}
    \hline
    Name & Global Minimum & Search Domain \\ 
    \hline
    Rastrigin Function & $f(0,\ldots,0) = 0$ & $-5.12 \leq x_i \leq 5.12$  \\
    Ackley Function & $f(0,0) = 0$ & $-5 \leq x,y \leq 5$ \\ 
    Sphere Function & $f(0,\ldots,0) = 0$ & $-\infty \leq x_i \leq \infty$ \\
    Rosenbrock Function & $f(1,\ldots,1) = 0$ & $-\infty \leq x_i \leq \infty$ \\
    Beale Function & $f(3,0.5) = 0$ & $-4.5 \leq x,y \leq 4.5$ \\
    Booth Function & $f(1,3)=0$ & $-10 \leq x,y \leq 10$ \\
    \hline
    \end{tabular}
\end{table}

\begin{table}[ht]
\centering
\caption{Table showing parameters set in the algorithm for these functions}
\label{values}
 \begin{tabular}{|c|c|p{25mm}|}
  \hline
  Function & Step-size & Lower bound \newline on the basin \\
  \hline
    Rastrigin Function & 0.0001 & 0.5  \\
    Ackley Function & 0.0001 & 0.1 \\ 
    Sphere Function & 0.001 & 0.3 \\
    Rosenbrock Function & 0.001 & 0.5 \\
    Beale Function & 0.0005 & 0.3 \\
    Booth Function & 0.005 & 0.3 \\
    \hline
  \end{tabular}
\end{table}

\begin{figure}[!ht]
\centering
\begin{subfigure}{.4\textwidth}
\includegraphics[width=.6\textwidth,angle=270]{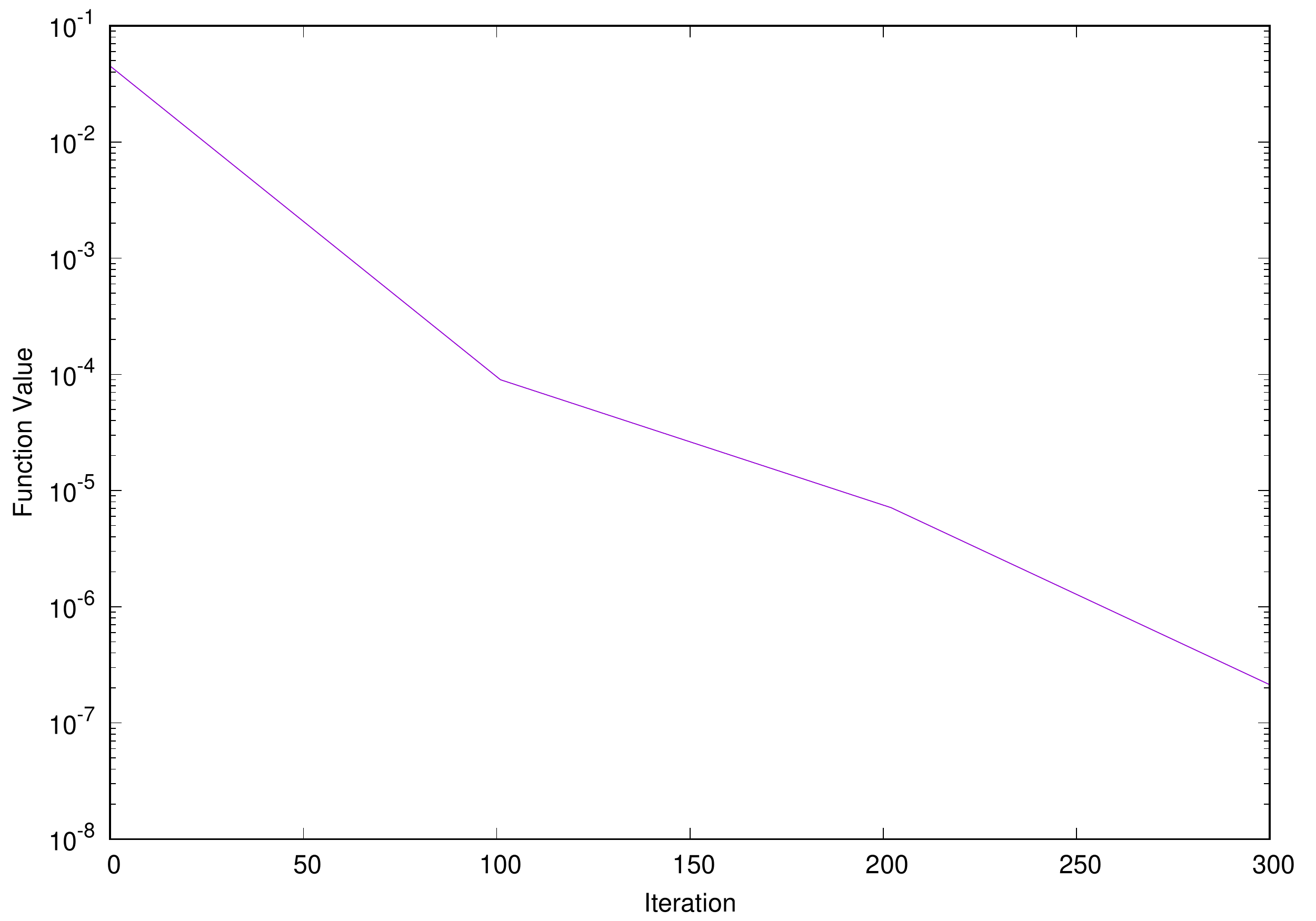}
\end{subfigure}
\begin{subfigure}{.4\textwidth}
\centering
\includegraphics[width=.7\textwidth,angle=270]{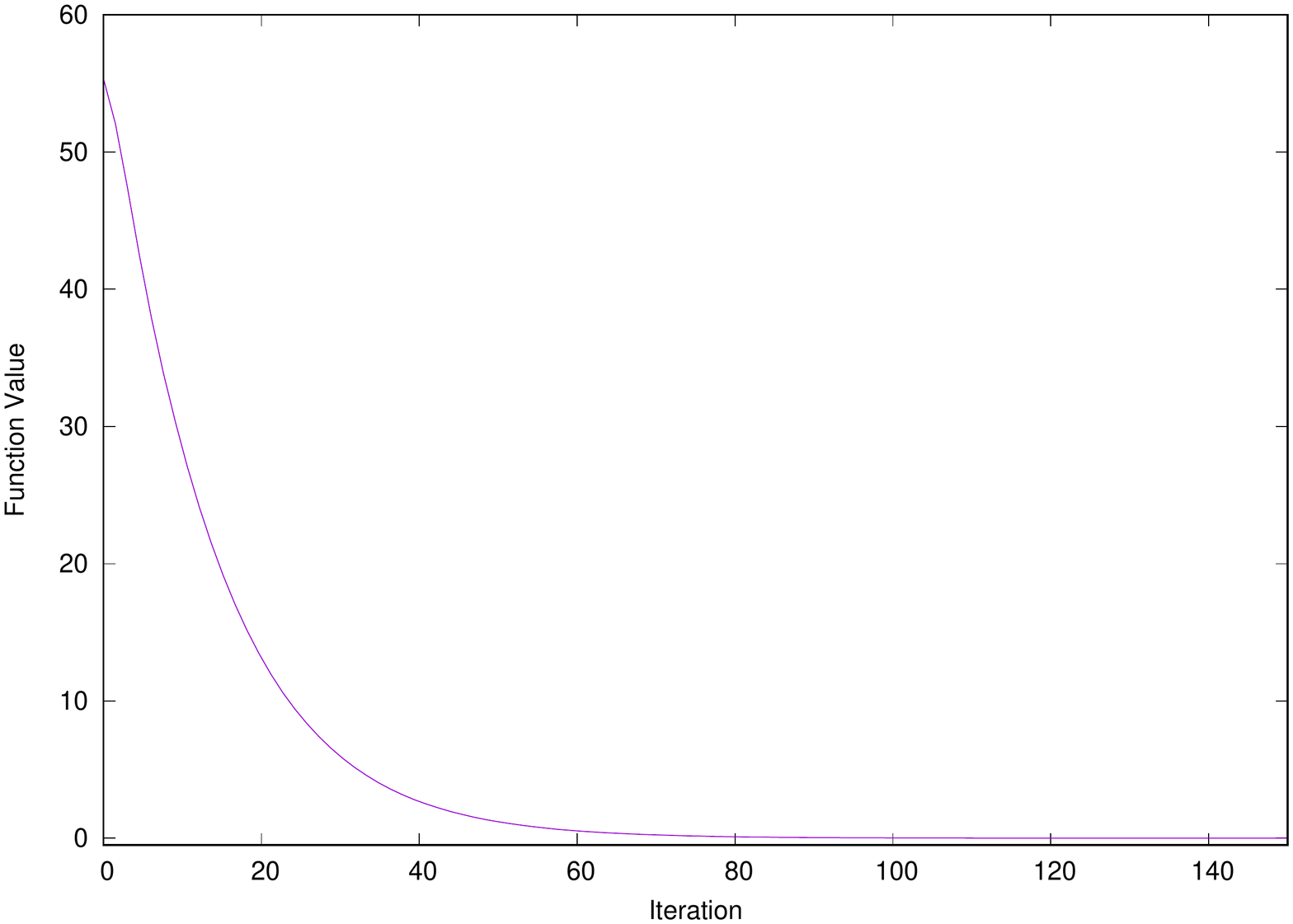}
\end{subfigure}
\caption{Convergence to Optimum for Ackley and Rastrigin Function}
\end{figure}

\begin{figure}[!ht]
\centering
\begin{subfigure}{.4\textwidth}
\includegraphics[width=.7\textwidth,angle=270]{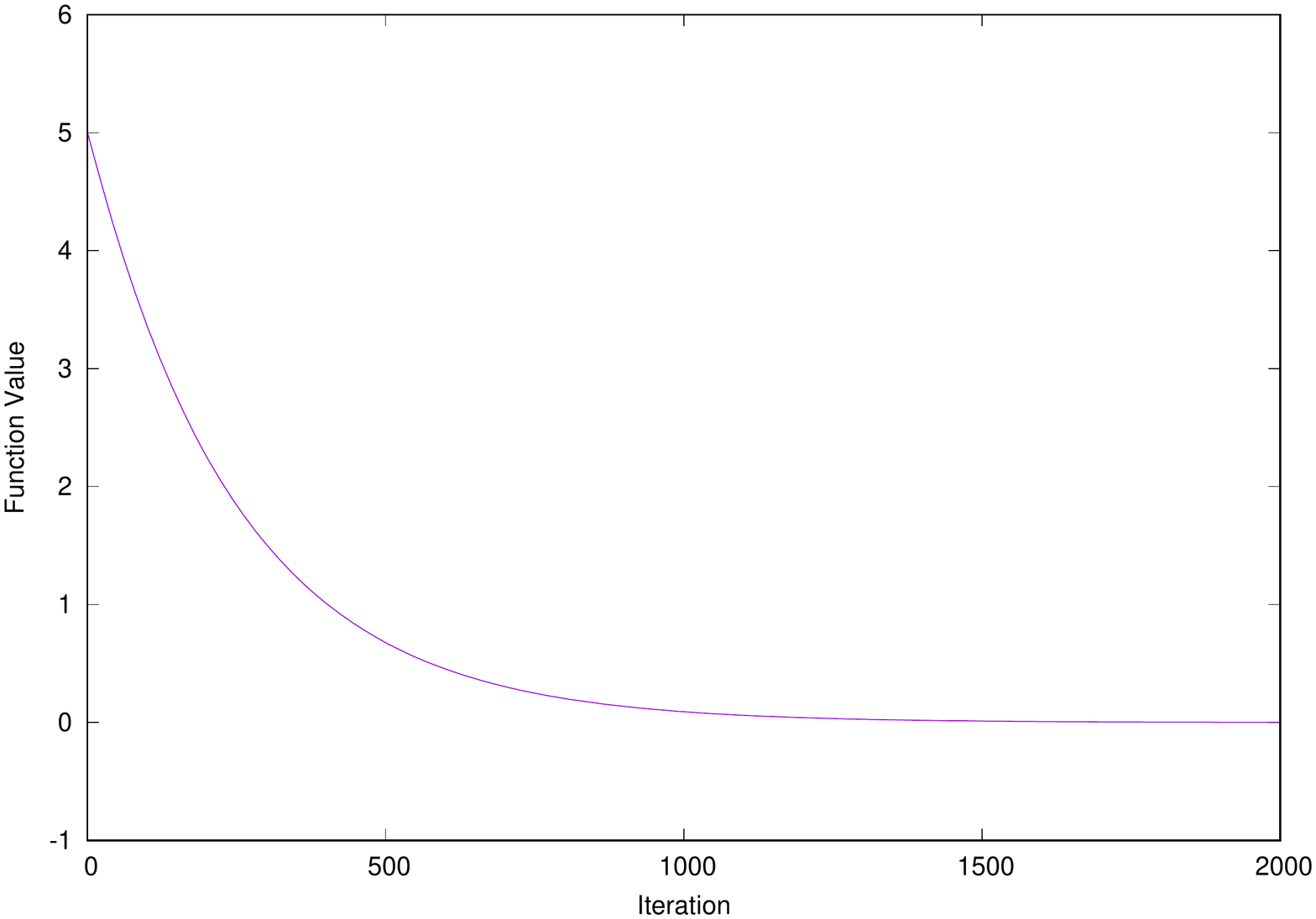}
\end{subfigure}
\begin{subfigure}{.4\textwidth}
\centering
\includegraphics[width=.6\textwidth,angle=270]{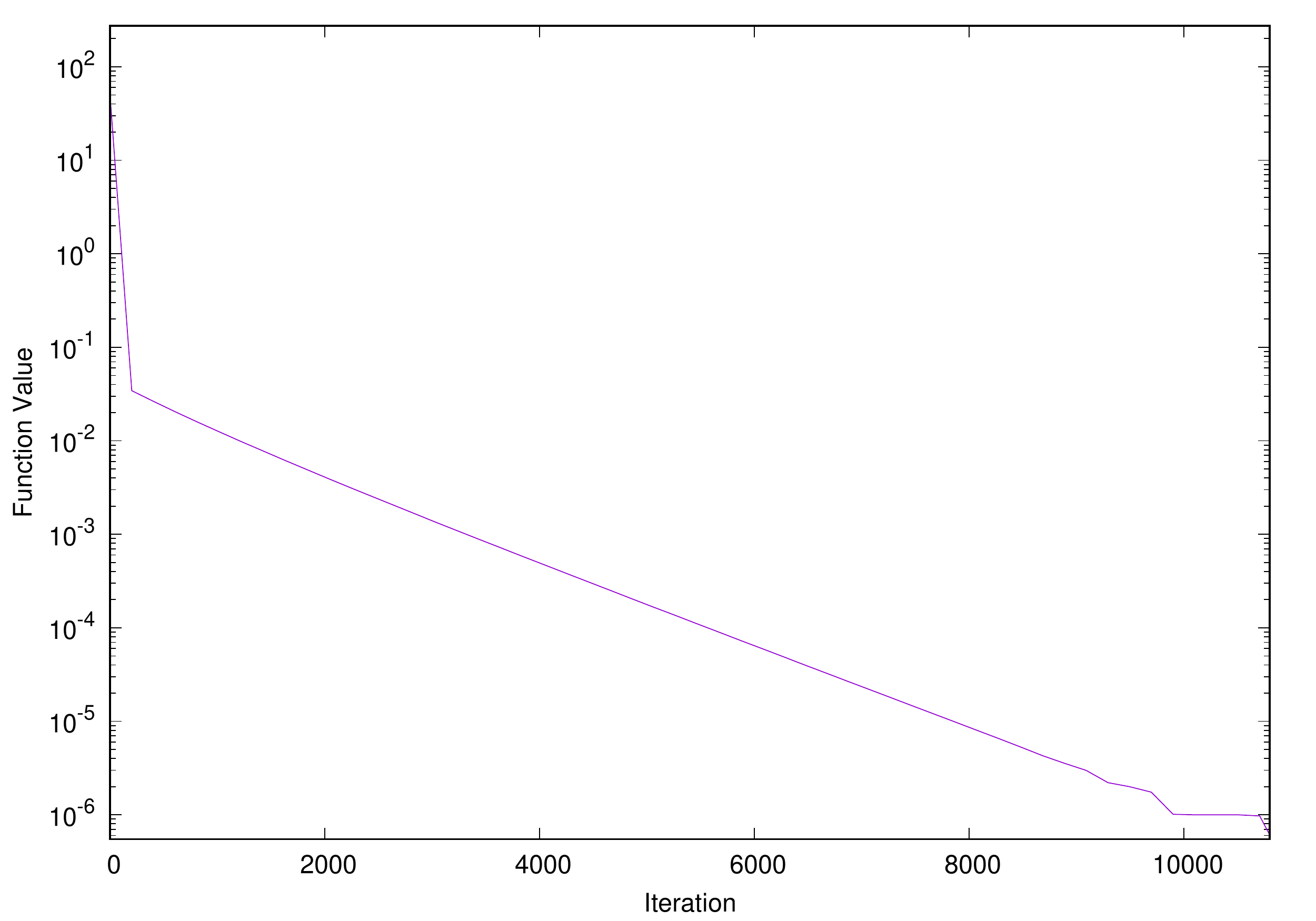}
\end{subfigure}
\caption{Convergence to Optimum for Sphere and Rosenbrock Function}
\end{figure}

\begin{figure}[!ht]
\centering
\begin{subfigure}{.4\textwidth}
\includegraphics[width=.6\textwidth,angle=270]{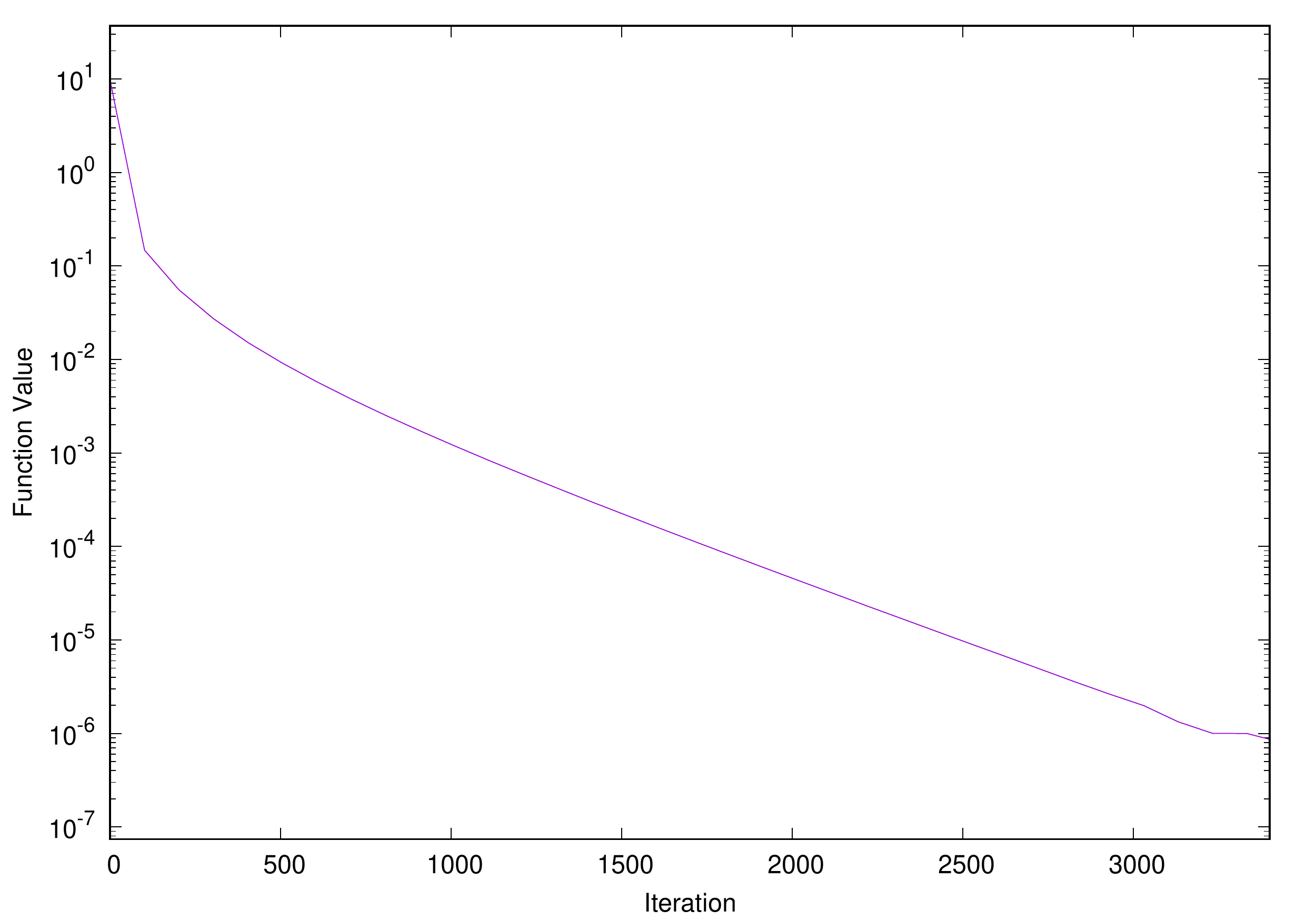}
\end{subfigure}
\begin{subfigure}{.4\textwidth}
\centering
\includegraphics[width=.6\textwidth,angle=270]{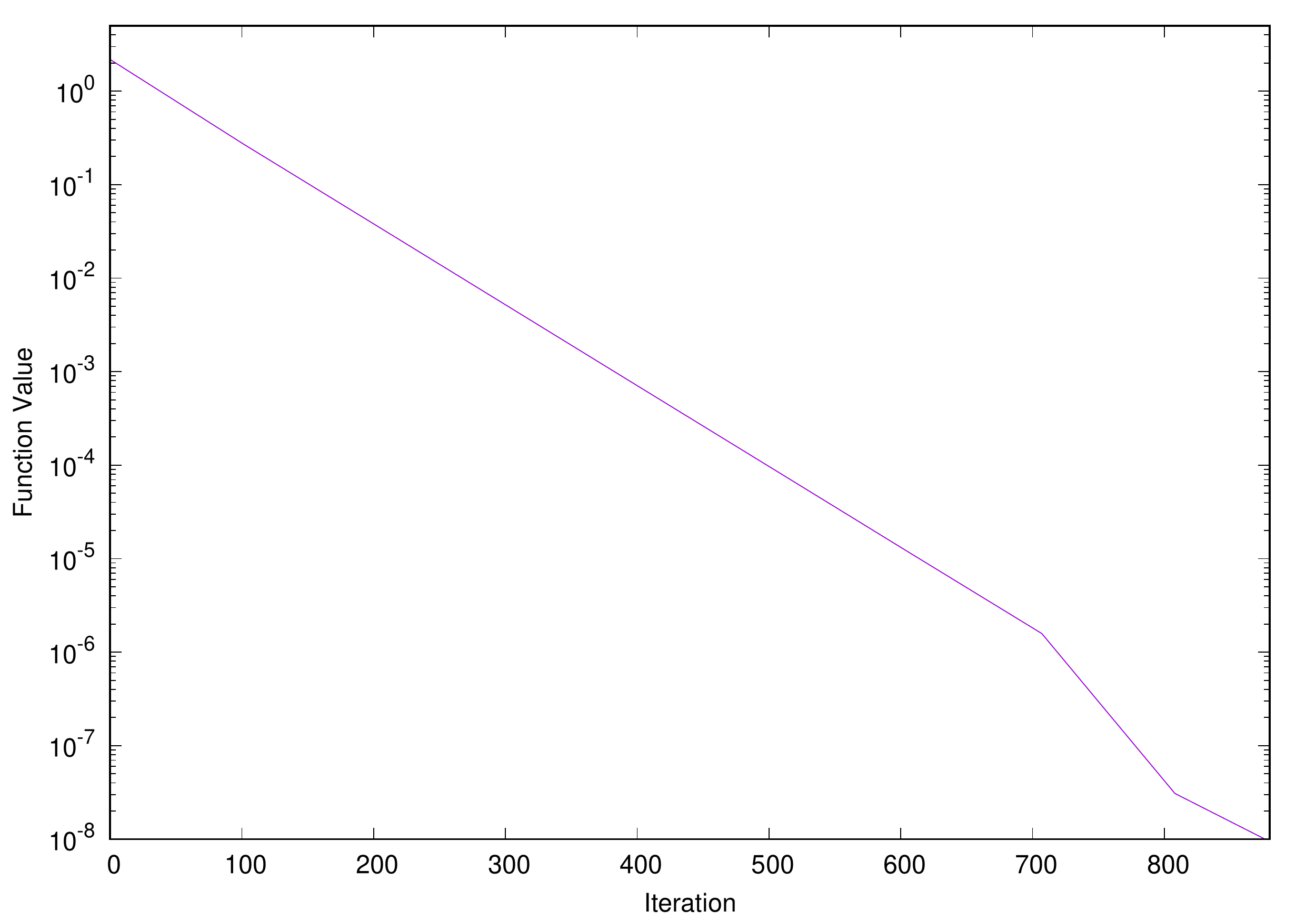}
\end{subfigure}
\caption{Convergence to Optimum for Beale and Booth Function}
\end{figure}

\section{Conclusion}
We have given a simple proof that finding global minima of a continuous function is not limit computable. To the best of our knowledge, this is the first such result. We have also given an algorithm that converges to the global minima when a lower bound to the basin of attraction of a global minimum is known. Finally, some numerical results were presented.

\bibliographystyle{plain}
\bibliography{limglobal}

\begin{thebibliography}{3}
\providecommand{\natexlab}[1]{#1}
\providecommand{\url}[1]{\texttt{#1}}
\expandafter\ifx\csname urlstyle\endcsname\relax
  \providecommand{\doi}[1]{doi: #1}\else
  \providecommand{\doi}{doi: \begingroup \urlstyle{rm}\Url}\fi

\bibitem[C. et~al.(2007)C., V., J.C., and J.]{helon}
D'Helon C., Protopopescu V., Wells J.C., and Barhen J.
\newblock {GMG} — a guaranteed global optimization algorithm: Application to
  remote sensing.
\newblock \emph{Mathematical and Computer Modelling}, 45\penalty0
  (3-4):\penalty0 459--472, 2007.

\bibitem[Horst and Tuy(1996)]{globook}
R.~Horst and H.~Tuy.
\newblock \emph{Global Optimization: Deterministic Approaches}.
\newblock Springer-Verlag, 1996.

\bibitem[Soare(2016)]{sorbook}
R.I. Soare.
\newblock \emph{Turing Computability: Theory and Applications}.
\newblock Springer-Verlag, 2016.

\end{thebibliography}

\end{document}